\RequirePackage{fix-cm}
\documentclass[envcountsect]{svjour3}                    

\smartqed  

\makeatletter
\def\@thmcountersep{.}
\makeatother

\usepackage{graphicx}%
\usepackage{amsmath,amssymb,amsfonts}%
\usepackage{mathrsfs}%

\DeclareMathOperator{\dist}{dist}
\DeclareMathOperator{\diam}{diam}
\def \mx {\mathbf{x}}
\def \my {\mathbf{y}}
\def \mz {\mathbf{z}}
\def \me {\mathbf{e}}

\begin{document}

\title{ Marcinkiewicz-Zygmund inequalities for hemispherical $t$-designs 
}

\titlerunning{Hemispherical $t$-designs and Marcinkiewicz-Zygmund inequalities}        

\author{Chao Li\textsuperscript{1}     \and
       Xiaojun Chen\textsuperscript{1}
}
\authorrunning{C. Li and X. Chen} 
\institute{ \textsuperscript{1}Department of Applied Mathematics, The Hong Kong Polytechnic University, Hong Kong, China.
lichao.li@polyu.edu.hk,
maxjchen@polyu.edu.hk
}

\date{}

\maketitle

\begin{abstract}

In this paper, we  introduce a new set of $\mathbb{L}_2$-orthogonal polynomials on the hemisphere $\mathbb{S}^d_+:=\{\mx\in\mathbb{R}^{d+1}:\|\mx\|=1, \mx\cdot\me_{d+1}\geq0\}$
and  establish $\mathbb{L}_1$ Marcinkiewicz-Zygmund inequalities on the hemisphere for even and odd spherical polynomials.
Hemispherical $t$-designs provide equal weight cubature rules on the hemisphere which are exact for polynomials up to degree $t$.
We propose variational characterizations of hemispherical $t$-designs by the $\mathbb{L}_2$-orthogonal polynomials.
Moreover, based on the Marcinkiewicz-Zygmund inequalities and the $\mathbb{L}_2$-orthogonal polynomials,  we  prove that for each $n\geq C_d t^d$, there exists a hemispherical $t$-design on the hemisphere with $n$ points, where $C_d$ is a constant depending only on $d$.
\keywords{spherical $t$-designs \and hemisphere \and Marcinkiewicz-Zygmund inequality \and even spherical polynomials \and cubature formulas}
\subclass{65D30 and 65D32}
\end{abstract}

\section{Introduction}
Let $\mathbb{S}^d:=\{\mx\in\mathbb{R}^{d+1}:\|\mx\|=1\}$ be the unit sphere in $\mathbb{R}^{d+1}$, where $\|\cdot\|$ is the Euclidean norm.
The problem of distributing a  number of points on the unit sphere  $\mathbb{S}^d$ has long  inspired  mathematical researchers, biologists, physicists and others.
Saff and  Kuijlaars \cite{saff1997} discussed in a general framework how to construct point sets on the unit sphere, their work has motivated and initiated a fruitful direction of research, see for example \cite{bannai2009,bondarenko2015,boyvalenkov2026,boyvalenkov2016,brauchart2015,sloan2026qmc} and references therein.
Distributing points on subsets of the sphere has also attracted attention in recent years and has many applications in sampling and approximation  on spherical caps \cite{feng2024,hesse2021} and spherical polygons\cite{cavoretto2026},
and geophysics  \cite{baratchart2017}.

Marcinkiewicz-Zygmund inequalities  have important applications of distributing point sets on the sphere for positive cubature formulas \cite{bondarenko2013,mhaskar2001,narcowich2006}.
Dai and Wang \cite{dai2010} proved the  Marcinkiewicz-Zygmund inequalities on spherical caps by introducing a new metric, which can be used to construct cubature formulas on spherical caps  as well as spherical collars.
In this work, we prove the classical $\mathbb{L}_1$ Marcinkiewicz-Zygmund inequalities \cite{mhaskar2001} on the hemisphere for even and odd spherical polynomials,  and show their application in establishing asymptotic bounds for hemispherical $t$-designs, that is, point sets on the hemisphere which provide equal weight cubature rules with polynomial exactness $t$.

Let $S\subset\mathbb{R}^{d+1}$ be a path-connect  topological space provided with a measure $\sigma$ that is finite and positive with full support, and let
$g_1,\ldots,g_m:S\rightarrow\mathbb{R}^p$ be continuous, integrable functions.
 An averaging set for $g_1,\ldots,g_m$ is a finite set $\mathcal{X}_n:=\{\mx_1,\ldots,\mx_n\}\subseteq S$ having the property
\begin{equation*}\label{gg}
 \frac{1}{\sigma(S)}\int_{S}g_j(\mathbf{x})d\sigma(\mx)=\frac{1}{n}\sum_{i=1}^ng_j(\mathbf{x}_i),\quad \mathrm{for}\ j=1,\ldots,m.
\end{equation*}
Seymour and Zaslavsky \cite{seymour1984} proved the existence of averaging sets which is stated as follows.
\begin{theorem}[Seymour and Zaslavsky \cite{seymour1984}]\label{SZ}
  Given $S$, $\sigma$, and $g_1,\ldots,g_m$ as described above, there exist averaging sets $\mathcal{X}_n$. The size of $\mathcal{X}_n$ may be any number, with a finite number of exceptions.
  And $\mathcal{X}_n$ may be chosen so that the vectors $(g_1(\mx_i),\ldots,g_m(\mx_i))^\top$ for all  $\mx_i\in \mathcal{X}_n$ are all distinct.
\end{theorem}

Let $\mathbb{P}_t$ be the space of spherical polynomials of degree at most $t$.
If we take $S=\mathbb{S}^d$ and $\{g_1,\ldots,g_m\}$ to be the set of spherical harmonics of degree at most $t$ on $\mathbb{S}^{d}$,
then the averaging set $\mathcal{X}_n:=\{\mathbf{x}_1,\ldots,\mathbf{x}_n\}\subseteq \mathbb{S}^d$ is called a \emph{spherical $t$-design}, that is, $$\frac{1}{|\mathbb{S}^d|}\int_{\mathbb{S}^d}P(\mathbf{x})d\omega_d(\mathbf{x})=\frac{1}{n}\sum_{i=1}^{n}P(\mathbf{x}_i),\quad \forall P\in\mathbb{P}_{t}.$$
where $d\omega_d(\mx)$ is the surface measure on $\mathbb{S}^{d}$ and $|\mathbb{S}^{d}|$ is the surface measure of the unit sphere.
The concept of spherical $t$-designs was introduced by Delsarte, Goethals and Seidel \cite{delsarte1977} in 1977.
And Theorem \ref{SZ}  gives the existence of spherical $t$-designs.
The exploration of the value of $n$ has persistently garnered significant attention.
Bannai and Damerell \cite{eiichi1979,eiichi1980} studied  spherical $t$-designs  for  $d\geq2$  with specific $t$.
For $d=2$, Hardin and Sloane \cite{hardin1996}  conjectured that $n\leq \frac{1}{2}t^2(1+o(1))$ as $t\rightarrow\infty$ and gave numerical evidence in support of this conjecture.
Chen and  Womersley \cite{chen2006},
Chen, Frommer, and Lang \cite{chen2011} proved that spherical $t$-design exists with $n=(t+1)^2$   for $t\leq 100$ on $\mathbb{S}^2$ by the Brouwer fixed point theorem and solving a system of nonlinear equations using the interval method.
Bondarenko and Viazovska \cite{bondarenko2010} showed that for each $n\geq c_d t^{\frac{2d(d+1)}{d+2}}$, there exists a spherical $t$-design on $\mathbb{S}^d$,   and
Bondarenko, Radchenko and Viazovska \cite{bondarenko2013} proved that for each $n\geq c_dt^d$, spherical $t$-designs with $n$ points exist,  where $c_d$ is a constant depending only on $d$.

In recent years, integration rules on subset of the sphere have attracted growing attention.
Mhaskar \cite{mhaskar2004,mhaskar2004local} proved that  positive weight cubature rules  on  spherical collars exist. For $d=2$,
Dai and Wang \cite{dai2010}  showed the existence of   positive weight cubature rules of precision $t$  with $O(t^2)$ points  on spherical zones under certain conditions.
Then, Hesse and Womersley \cite{hesse2012} presented  spherical cap $t$-designs with $O(t^3)$ points.
In \cite{chen2024}, we introduced a set of points named spherical cap $t$-subdesign induced by a spherical $t$-design for numerical integration of zonal spherical polynomials.
More recently, we \cite{chen2025} constructed spherical zone $t$-designs with $O(t^3)$ points.

Let  $\mathbb{S}_+^d:=\{\mx\in\mathbb{S}^d:\mx\cdot\me_{d+1}\geq0\}$ be the upper hemisphere, where $\me_{d+1}=(0,\ldots,0,1)^\top\in\mathbb{R}^{d+1}$ is the north pole of the sphere.
A finite set $\mathcal{X}_n=\{\mx_1,\ldots,\mx_n\}\subseteq \mathbb{S}^d_+$ is called a \emph{hemispherical $t$-design} if it holds that
\begin{equation}\label{eqw}
\frac{1}{|\mathbb{S}^d_+|}\int_{\mathbb{S}^d_+}P(\mx)d\omega_d(\mx)=\frac{1}{n}\sum_{i=1}^{n}P(\mx_i), \quad \forall P\in\mathbb{P}_{t},
\end{equation}
where  $|\mathbb{S}^d_+|$ is the surface area of the hemisphere.
Taking $S=\mathbb{S}^d_+$ and $\{g_1,\ldots,g_m\}$ to be a set of orthogonal polynomials of degree at most $t$ on $\mathbb{S}^d_+$, Theorem \ref{SZ} shows
the existence of hemispherical $t$-designs, that is, for each pair of positive integers $d$ and $t>0$, and for all sufficiently large $n$,  hemispherical $t$-designs with $n$ points exist.

In this work, we will  give equivalent conditions for hemispherical $t$-designs by introducing a  new set of $\mathbb{L}_2$-orthogonal polynomials  and establish the  following theorem on optimal asymptotic bounds for  hemispherical $t$-designs based on these conditions and the aforementioned classical $\mathbb{L}_1$ Marcinkiewicz-Zygmund inequalities on the hemisphere.
\begin{theorem}\label{main}
  For each $n\geq C_dt^d$, there exists a hemispherical $t$-design on the hemisphere $\mathbb{S}_+^d$ consisting of $n$ points, where $C_d$ is a constant depending only on $d$.
\end{theorem}

The proof of  Theorem \ref{main}
is based on the following result of the Brouwer degree theorem (see Theorems 1.2.6 and 1.2.9 in \cite{cho2006}).
\begin{theorem}\cite{cho2006}\label{lemm:mapp}
Let $f:\mathbb{R}^n\rightarrow\mathbb{R}^n$ be a continuous mapping and $\Omega$ be an open bounded subset with boundary $\partial\Omega$ such that $0\in\Omega\subset \mathbb{R}^n$.
If $\langle \mathbf{s},f(\mathbf{s}) \rangle>0$ for all $\mathbf{s}\in\partial\Omega$, then there exists $\bar{\mathbf{s}}\in\Omega$ such that $f(\bar{\mathbf{s}})=0$.
\end{theorem}

From  \cite{bondarenko2013}, we observe that the construction of  $f$ in Theorem \ref{lemm:mapp} relies on auxiliary results on the sphere, including equivalent conditions for spherical $t$-designs and  equal area partitions of $\mathbb{S}^d$.
Besides, the $\mathbb{L}_1$ Marcinkiewicz-Zygmund inequalities on $\mathbb{S}^d$ are necessary for proving the condition in Theorem \ref{lemm:mapp}, see \cite{bondarenko2013} for more details.
However, the aforementioned auxiliary results on $\mathbb{S}^d$  cannot be directly extended to  $\mathbb{S}_+^d$.
For example, the spherical harmonics are not orthogonal on subsets of $\mathbb{S}^d$, therefore, we can not apply them to characterize the hemispherical $t$-design in the desired form.
Besides, the classical $\mathbb{L}_1$ Marcinkiewicz-Zygmund inequalities on subsets of the sphere do not hold.
These make our work nontrivial.

By the rotationally invariance property and  Lemma 3.2 in \cite{chen2025}, it is sufficiency to consider hemispherical $t$-designs on the upper hemisphere $\mathbb{S}_+^d$.

The main contributions of this paper are as follows.
\begin{itemize}
  \item We define a new set of spherical polynomials $\{B_{\ell,k}^d\} $ which are $\mathbb{L}_2$-orthonormal on $\mathbb{S}_+^d$, and present equivalent conditions for hemispherical $t$-designs.
   \item   We  prove the  $\mathbb{L}_1$ Marcinkiewicz-Zygmund inequalities on the hemisphere for even and odd spherical polynomials.
   \item Based on the $\mathbb{L}_1$ Marcinkiewicz-Zygmund inequalities and the equivalent conditions, we prove Theorem \ref{main}.
       \end{itemize}

The rest of this paper is organised as follows.
In section 2, we introduce a new set of spherical polynomials  and propose variational characterizations for hemispherical $t$-designs.
In section 3, we  prove the $\mathbb{L}_1$ Marcinkiewicz-Zygmund inequalities on  the hemisphere for even and odd spherical polynomials.
In section 4, we  prove Theorem \ref{main} and present a hemispherical 3-design with 8 and 16 points in section 5.
 We conclude this paper in section 6.

\section{Equivalent conditions for  hemispherical $t$-designs}
\subsection{Notation}
 Let $\mathbb{N}:=\{1,2,3,\ldots\}$ be the set of natural numbers and $\mathbb{N}_0:=\{0\}\cup\mathbb{N}$.
We denote by $\mathbb{P}_t^e$ and $\mathbb{P}_t^o$ the spaces of even and odd spherical polynomials of degree at most $t$, respectively.

The spherical polar coordinates of $\mx=(x_1,x_2,\ldots,x_{d+1})^\top\in\mathbb{S}^d$  is given by
\begin{equation}\label{polar}
\left\{
\begin{array}{rl}
   x_{1}=  & \sin\theta_{d-1}...\sin\theta_{1}\cos\phi\\
    x_{2} =&  \sin\theta_{d-1}\ldots\sin\theta_{1}\sin\phi\\
    \vdots &\\
  x_d =& \sin\theta_{d-1}\cos\theta_{d-2} \\
  x_{d+1} =& \cos\theta_{d-1}, \\
\end{array}
\right.
\end{equation}
where  $\phi\in[0, 2\pi]$ and $\theta_i\in[0,\pi]$, $i=1,\ldots,d-1$.
It is easy to see that
$\mx=(\bar{\mx}\sin\theta_{d-1},\cos\theta_{d-1})^\top\in\mathbb{S}^d$, where $\theta_{d-1}\in[0,\pi]$ and $\bar{\mx}\in\mathbb{S}^{d-1}$.
And for a continuous function $f$, there holds
$$\int_{\mathbb{S}^d}f(\mx)d\omega_d(\mx)=\int_{0}^{\pi}\int_{\mathbb{S}^{d-1}}f(\bar{\mx}\sin\theta_{d-1},\cos\theta_{d-1})d\omega_{d-1}(\bar{\mx})(\sin\theta_{d-1})^{d-1}d\theta_{d-1},$$
where  $d\omega_d(\mx)=\prod_{i=1}^{d-1}(\sin\theta_{d-i})^{d-i} d\theta_{d-1}\ldots d\theta_1d\phi$.
And the surface measure of the unit sphere is $|\mathbb{S}^{d}|=\frac{2\pi^{(d+1)/2}}{\Gamma(\frac{d+1}{2})}$, where $\Gamma$ is the Gamma function.

For convenience, for $d\geq 2$ and $\ell\in\mathbb{N}_0$,  let
$$\mathcal{I}_{\ell}^d:=\{
k=(k_{1},\ldots,k_{d-1})\in\mathbb{N}_0\times\ldots\times\mathbb{N}_0\times\mathbb{Z}:
|k_{d-1}|\leq k_{d-2}\leq\ldots\leq k_1\leq\ell\}.$$
And for $k\in\mathcal{I}_{\ell}^d$, $k=0$ means $k_1=k_2=\ldots=k_{d-1}=0$.

\subsection{Polynomials on the hemisphere}
Orthogonal polynomials on  certain domains are essential for research, see for example \cite{dai2013,ku2005}.
Since the spherical harmonics are not orthogonal on subsets of the sphere, we define new polynomials which are $\mathbb{L}_2$-orthonormal on the hemisphere.

We first apply the Gram-Schmidt orthogonalization to define orthogonal polynomials on $[0,1]$.
Let $\lambda_i=i+\frac{d-2}{2}$, $i\in\mathbb{N}_0$, we define
$f_{r,i}=\int_{0}^{1}x^r(1-x^2)^{\lambda_i} dx$, $r\in\mathbb{N}_0$. 
Then, let $p_{0,0}=1$ and for $j>0$, let
$$p_{j,i}(x)=\hspace{-1mm}\frac{1}{\Delta_{j-1}^{i}}\left|
  \begin{array}{cccc}
    f_{0,i} & f_{1,i} & \cdots & f_{j,i}\\
    f_{1,i} & f_{2,i} & \cdots & f_{j+1,i} \\
    \vdots & \vdots & \ddots & \vdots \\
    f_{j-1,i} & f_{j,i} & \cdots & f_{2j-1,i} \\
    1 & x & \cdots & x^{j}
  \end{array}
\right|, \mathrm{where}\,
 \Delta_{j-1}^{i}=\hspace{-1mm}
\left|
  \begin{array}{cccc}
    f_{0,i}  & \cdots & f_{j-1,i}\\
    f_{1,i}  & \cdots & f_{j,i} \\
    \vdots  & \ddots & \vdots \\
    f_{j-1,i}  & \cdots & f_{2j-2,i}
  \end{array}
\right|,$$
and $|\cdot|$ is the deteminnate of a matrix.

Finally, for $i\in\mathbb{N}_0$, we obtain
$$\Psi_{j,i}(x)=\frac{p_{j,i}(x)}{(\int_{0}^1 (p_{j,i}(x))^2(1-x^2)^{\lambda_i}dx  )^{1/2}}, \quad \forall x\in[0,1].$$
 And for each fixed $i$,
\begin{equation}\label{or:psi}
  \int_{0}^1\Psi_{j,i}(x)\Psi_{j',i}(x)(1-x^2)^{\lambda_i}dx=\delta_{jj'}.
\end{equation}

Now, combing $\Psi_{j,i}$ and spherical harmonics $Y_{k}^{d-1}$ on $\mathbb{S}^{d-1}$,  we define the orthogonal polynomials on the hemisphere $\mathbb{S}^d_+$.

\begin{definition}
For $d\geq 2$, $\ell\in\mathbb{N}_0$ and $k=(k_1,\ldots,k_{d-1})\in\mathcal{I}_{\ell}^d$,
we define polynomials on the hemisphere $\mathbb{S}^d_+$ as follows,
$$B^d_{\ell,k}(\xi,\theta_{d-1})=(\sin\theta_{d-1})^{k_1}\Psi_{\ell-k_1,k_1}(\cos\theta_{d-1})Y^{d-1}_{k}(\xi),
$$
where $\theta_{d-1}\in[0,\frac{\pi}{2}]$ and $\xi=(\phi,\theta_1,\ldots,\theta_{d-2})\in[0,2\pi]\times[0,\pi]^{d-2}$, $Y_{k}^{d-1}$ are spherical harmonics on $\mathbb{S}^{d-1}$.
\end{definition}
We will write $B_{\ell,k}^d(\mx):=B_{\ell,k}^d(\xi,\theta_{d-1})$ with $\mx\in\mathbb{S}_+^d$ satisfying (\ref{polar}).

Now we prove the orthogonality of $B^d_{\ell,k}$.

\begin{proposition}\label{pro:B}
  For any $\ell,\ell'\in\mathbb{N}_0$, and $k\in\mathcal{I}_{\ell}^d,k'\in\mathcal{I}_{\ell'}^d$,
   there holds
   \begin{equation}\label{or:B}
   \int_{\mathbb{S}^d_+}B^d_{\ell,k}(\mx)B^d_{\ell',k'}(\mx)d\omega_d(\mx)=\delta_{\ell\ell'}\delta_{kk'}.
  \end{equation}
\end{proposition}
\begin{proof}
Since $\int_{\mathbb{S}^{d-1}}Y^{d-1}_{k}(\xi)Y^{d-1}_{k'}(\xi)d\omega_{d-1}(\xi)=\delta_{kk'}$,
by the definition of $B^d_{\ell,k}$, we obtain that
\begin{equation*}
\begin{split}
  & \int_{\mathbb{S}^d_+}B^d_{\ell,k}(\mx)B^d_{\ell',k'}(\mx)d\omega_d(\mx) \\ &=\delta_{kk'}\int_{0}^{\frac{\pi}{2}}(\sin\theta_{d-1})^{2k_1}\Psi_{\ell-k_1,k_1}(\cos\theta_{d-1})\Psi_{\ell'-k_1,k_1}(\cos\theta_{d-1})(\sin\theta_{d-1})^{d-1}d\theta_{d-1}
  \\
     & =\delta_{kk'}\delta_{\ell\ell'},
\end{split}
   \end{equation*}
where the last equality follows from (\ref{or:psi}).
The proof is completed.
\end{proof}

By Proposition \ref{pro:B}, we obtain the following property of $B_{\ell,k}^d$ on the hemisphere.

\begin{proposition}\label{pro:intB}
  For  $\ell\in\mathbb{N}$ and $k\in\mathcal{I}_{\ell}^d$,
   there holds
   $$\int_{\mathbb{S}^d_+}B^d_{\ell,k}(\mx)d\omega_d(\mx)=0.$$
\end{proposition}

\begin{proof}
Since $B^d_{0,0}=\Psi_{0,0}Y_0^{d-1}=\sqrt{\frac{2\Gamma(\frac{d+1}{2})}{\sqrt{\pi}\Gamma(\frac{d}{2})}}\frac{1}{\sqrt{|\mathbb{S}^{d-1}|}}$ is a constant, by Proposition \ref{pro:B},
we obtain
$$
\int_{\mathbb{S}^d_+}B^d_{\ell,k}(\mx)B^d_{0,0}(\mx)d\omega_d(\mx)
   =0,\quad \forall \ell\neq0,
$$
which implies $\int_{\mathbb{S}^d_+}B^d_{\ell,k}(\mx)d\omega_d(\mx)=0,$ $\forall \ell\neq0$.
The proof is completed.
\end{proof}

\subsection{Equivalent conditions for  hemispherical $t$-designs}
In this section, we present  conditions for characterizing   hemispherical  $t$-designs based on the newly defined polynomials $B_{\ell,k}^d$.

\begin{theorem}\label{con}
 The set $\mathcal{X}_n=\{\mx_1,\ldots,\mx_n\}\subseteq \mathbb{S}^d_+$ is a  hemispherical $t$-design if and only if
 \begin{equation}\label{the:iif}
   \sum_{i=1}^{n}B^d_{\ell,k}(\mx_i)=0,\quad \forall k\in\mathcal{I}_{\ell}^d ,\ \ell=1,2,\ldots,t.
 \end{equation}
\end{theorem}

\begin{proof}
  Since $\mathcal{X}_n$ is a  hemispherical  $t$-design, we have
  $$\int_{\mathbb{S}^d_+}B^d_{\ell,k}(\mx)d\omega_{d}(\mx)=\frac{|\mathbb{S}^d_+|}{n}\sum_{i=1}^{n}B^d_{\ell,k}(\mx_i),\quad \forall k\in\mathcal{I}_{\ell}^d ,\ \ell=0,1,2,\ldots,t$$
Then, by Proposition \ref{pro:intB}, (\ref{the:iif}) holds.

Let (\ref{the:iif}) hold.
For any spherical polynomial $P\in\mathbb{P}_t$, there are $\alpha_{\ell,k}\in\mathbb{R}$ such that
$P=\sum_{\ell=0}^{t}\sum_{k\in\mathcal{I}^d_{\ell}}\alpha_{\ell,k}B^d_{\ell,k}$.
Then, we have
\begin{equation*}
  \begin{split}
     \sum_{i=1}^{n}P(\mx_i) & =\sum_{i=1}^{n}\sum_{\ell=0}^{t}\sum_{k\in\mathcal{I}^d_{\ell}}\alpha_{\ell,k}B^d_{\ell,k}(\mx_i)
     = \sum_{\ell=0}^{t}\sum_{k\in\mathcal{I}^d_{\ell}}\alpha_{\ell,k}\sum_{i=1}^{n}B^d_{\ell,k}(\mx_i)\\
       & =\alpha_{0,0}\sum_{i=1}^{n}B^d_{0,0}(\mx_i)
       =\alpha_{0,0}\frac{n}{|\mathbb{S}^d_+|}\int_{\mathbb{S}^d_+}B^d_{0,0}(\mx)d\omega_d(\mx)\\
       &=\frac{n}{|\mathbb{S}^d_+|}\sum_{\ell=0}^{t}\sum_{k\in\mathcal{I}^d_{\ell}}\alpha_{\ell,k}\int_{\mathbb{S}^d_+}B^d_{\ell,k}(\mx)d\omega_d(\mx)
       =\frac{n}{|\mathbb{S}^d_+|}\int_{\mathbb{S}^d_+}P(\mx)d\omega_d(\mx),
  \end{split}
\end{equation*}
where the fifth equality follows from Proposition \ref{pro:intB}.
Thus, $\mathcal{X}_n$ is a hemispherical $t$-design.
The proof is completed.
\end{proof}

Let $\mathbb{P}_t^*:=\{P\in\mathbb{P}_t:\int_{\mathbb{S}^d_+}P(\mx)d\omega_d(\mx)=0\}$ be equipped with the usual inner product
$\langle P,Q\rangle=\int_{\mathbb{S}^d_+}Q(\mx)P(\mx)d\omega_d(\mx)$, $\forall P,Q\in\mathbb{P}_t^*$.
By Theorem  \ref{con}, the following corollary holds.
\begin{corollary}\label{coro:con}
The set
$\mathcal{X}_{n}=\{\mx_1,\ldots,\mx_n\}\subseteq\mathbb{S}^d_+$
 is a hemispherical $t$-design on $\mathbb{S}^d_+$ if and only if
 $\sum_{i=1}^{n}P(\mx_i)=0,$ $\forall P\in\mathbb{P}_t^*.$
\end{corollary}

It is easy to see that for each point $\mx_0\in\mathbb{S}^d_+$, there exists a  polynomial
$Q_{\mx_0}\in\mathbb{P}^*_t$ represented as
\begin{equation}\label{poly:Q}
  Q_{\mathbf{x}_0}:=\sum_{\ell=1}^t\sum_{k \in\mathcal{I}_{\ell}^d}B^d_{\ell,k}(\mx_0)B^d_{\ell,k}.
\end{equation}
Moreover, for any $P\in\mathbb{P}^*_t$, there are $\alpha_{\ell,k}\in\mathbb{R}$ such that  $P=\sum_{\ell=1}^{t}\sum_{k\in\mathcal{I}_\ell^d}\alpha_{\ell,k}B_{\ell,k}^d$.
Thus, by Proposition \ref{pro:B},  there holds
\begin{equation}\label{poly:Q1}
  \langle Q_{\mx_0}, P \rangle=\int_{\mathbb{S}^d_+}Q_{\mathbf{x}_0}(\mx)P(\mx)d\omega_d(\mx)=P(\mx_0),\quad \forall P\in\mathbb{P}^*_t.
\end{equation}
Then,
a point set $\{\mx_1,\ldots,\mx_n\}\subseteq\mathbb{S}^d_+$ forms a hemispherical $t$-design if and only if
\begin{equation}\label{iffQ}
  Q_{\mx_1}+\ldots+Q_{\mx_n}=0,\quad  Q_{\mx_i}\in\mathbb{P}_t^*,\ i=1,\ldots,n,
\end{equation}
where the necessity follows from  Theorem \ref{con},
the sufficiency is due to Corollary \ref{coro:con} and  (\ref{poly:Q1}) that
$\sum_{i=1}^{n}\langle Q_{\mx_i}, P \rangle=\sum_{i=1}^{n} P(\mx_i)=0$ for any $P\in\mathbb{P}^*_t$.

\section{Marcinkiewicz-Zygmund inequalities on the hemisphere}
In this section, we prove the $\mathbb{L}_1$-Marcinkiewicz-Zygmund inequalities on the hemisphere for even and odd spherical polynomials.
We first show equal area partitions of the hemisphere.

\subsection{Equal area partitions}
Let $\mathcal{R}=\{R_1,\ldots, R_n\}$ be an equal area partition of the hemisphere $\mathbb{S}^d_+$, where $\omega_d(R_i)=\frac{|\mathbb{S}^d_+|}{ n}$, $\cup_{i=1}^nR_i=\mathbb{S}^d_+$  and $\omega_d(R_i\cap R_j)=0$ for all $1\leq i<j\leq n$.
The partition norm for $\mathcal{R}$ is defined as follows
$$\|\mathcal{R}\|:=\max_{R\in\mathcal{R}}\diam R,$$
where  $\diam R:=\sup\{\arccos(\mx\cdot \my):\mx,\my\in R\}$
is the diameter of a region $R$.

Utilizing the partition method  in section 3 in \cite{leopardi2006}, we can always obtain  a recursive zonal  equal area partition of the hemisphere $\mathbb{S}^d_+$ into $n$ regions with bounded partition norm.
Moreover, following a similar proof of Theorem 1.6 in [3], we obtain the following result.
\begin{theorem}\label{theo:eqb}
  For any $n\in\mathbb{N}$, there is an equal area partition $\mathcal{R}=\{R_1,\ldots, R_n\}$ of the hemisphere $\mathbb{S}^d_+$ with
  $\|\mathcal{R}\|\leq c_{d} n^{-\frac{1}{d}}$, where   $c_{d}$ is a constant depending only on  $d$.
\end{theorem}

\subsection{Marcinkiewicz-Zygmund inequalities on $\mathbb{S}_+^d$}

We first present some existing results on doubling weights.

\begin{proposition}\cite{erdelyi1999}\label{erdelyi1999}
For any $x\in\mathbb{R}$, let $W_\delta(x):=\frac{\delta}{2}\int_{x-1/\delta}^{x+1/\delta}W(s)ds$, where  $W(s)=|\sin s|$, $\delta>0$,  and
$V$ be a trigonometric polynomial.\\
(i) For every $x,y\in\mathbb{R}$,
\begin{equation}\label{eq:W1}
  \frac{W_\delta(x)}{(2+2\delta|x-y|)^2}\leq W_{\delta}(y)\leq (2+2\delta|x-y|)^2W_\delta(x).
\end{equation}
(ii) If $\mathrm{deg}V\leq \delta$, then there is a constant $c_1$ such that
\begin{equation}\label{eq:W2}
  c_1^{-1}\int_{-\pi}^\pi|V(s)|W(s)ds\leq \int_{-\pi}^\pi|V(s)|W_\delta(s)ds \leq c_1\int_{-\pi}^\pi|V(s)|W(s)ds.
\end{equation}
(iii) There is a constant $c_2$ such that the following Bernstein inequalities hold,
\begin{equation}\label{eq:W3}
  \int_{-\pi}^\pi|V'(s)|W(s)ds\leq c_2\mathrm{deg}V\int_{-\pi}^\pi|V(s)|W(s)ds.
\end{equation}
\end{proposition}

\begin{lemma}\label{mhaskar2001}\cite{mhaskar2001}
Let $W$ be a doubling weight and $\delta>0$, $J$ be a closed interval, and $g:J\rightarrow[0,\infty)$ be integrable and have $J$ as its support.
If $|J|\leq 2\delta^{-1}$, then
\begin{equation}\label{eq:gW}
  \int_{J}g(s_1)ds_1  \int_{J}W(s_2)ds_2\leq \frac{2^{2q+1}}{\delta}\int_{J}g(s_1)W_\delta(s_1)ds_1.
\end{equation}
\end{lemma}

We prove the   following lemma in preparation for the proof of Marcinkiewicz-Zygmund inequalities.
For any $ \mathbf{x,y}\in\mathbb{S}^d$, let
\begin{equation}\label{Vt}
V_t(\mathbf{x\cdot y})=\sum_{\ell=0}^{2t-1}\frac{(2t-\ell)(2\ell+d-1)}{(d-1)|\mathbb{S}^d| t}G_{\ell}^{\frac{d-1}{2}}(\mx\cdot\my)-
\sum_{\ell=0}^{t-1}\frac{(t-\ell)(2\ell+d-1)}{(d-1)|\mathbb{S}^d| t}G_{\ell}^{\frac{d-1}{2}}(\mx\cdot\my),
\end{equation}
where $G_\ell^d$ are Gegenbauer polynomials.
Obviously, for any $P\in\mathbb{P}_t$, there holds $
P(\mathbf{y})=\int_{\mathbb{S}^d}V_t(\mathbf{x}\cdot\mathbf{y}) P(\mathbf{x})d\omega_d(\mathbf{x})$, $\forall \mathbf{y}\in\mathbb{S}^d.$
Moreover, $V_t$ is  a univariate polynomial of degree $2t-1$ and by the discussions in section 2 in \cite{mhaskar2001}, for any $t\geq1$,
$\|V_t\|_{\mathbb{L}_1}:=\int_{-1}^1|V_t(s)|ds$ is uniformly bounded in $t$.

\begin{lemma}\label{pro:bound}
Let $t\geq1$ and $V_t$ be defined as (\ref{Vt}).
For some $n\in\mathbb{N}$, let $\mathcal{R}=\{R_1,\ldots,R_n\}$ be an equal area partition  of the hemisphere $\mathbb{S}^d_+$ with partition norm $\|\mathcal{R}\|$ and
 $\mathbf{x}_i\in R_i$, $i=1,\ldots,n$.
 If $\|\mathcal{R}\|\leq(4t-2)^{-1}$, then
there is $c'>0$ such that
\begin{equation}\label{ineqV}
\sup_{\mathbf{y}\in\mathbb{S}^d}\sum_{i=1}^{n}\int_{R_i}|V_t(\mathbf{y}\cdot\mathbf{x}) -V_t(\mathbf{y}\cdot\mathbf{x}_i) |d\omega_d(\mathbf{x})
\leq c'(2t-1)C_v\|\mathcal{R}\|,
\end{equation}
where $C_v=\sup_{t\geq 1}\|V_t\|_{\mathbb{L}_1}$.
\end{lemma}
\begin{proof}
Suppose $\mathbf{y}_0=\arg \max_{\mathbf{y}\in\mathbb{S}^d}\sum_{i=1}^{n}\int_{R_i}|V_t(\mathbf{y}\cdot\mathbf{x}) -V_t(\mathbf{y}\cdot\mathbf{x}_i) |d\omega_d(\mathbf{x}).$
Let $\mathbf{R}$ be a rotation matrix such that $\mathbf{Ry}_0=\mathbf{e}_{d+1}$ and $\mathbb{S}^d_{\my_0}$ be the hemisphere rotated from $\mathbb{S}^d_+$ by $\mathbf{R}$.
Let $\mathcal{R}'=\{R_1',\ldots,R_n'\}$ be the corresponding equal area partition of $\mathbb{S}^d_{\my_0}$, we have $\|R\|=\|R'\|$.
Denote $\cos\vartheta=\mathbf{e}_{d+1}\cdot\mathbf{Rx}$, $\forall\mathbf{x}\in\mathbb{S}^d_+$,
$\cos\vartheta_i=\mathbf{e}_{d+1}\cdot\mathbf{Rx}_i$, $\mathbf{x}_i\in R_i$,
then
\begin{eqnarray}
  &&\mathcal{E}_i:=\int_{R_i}|V_t(\mathbf{y}_0\cdot\mathbf{x}) -V_t(\mathbf{y}_0\cdot\mathbf{x}_i) |d\omega_d(\mathbf{x})\nonumber\\
  &&= \int_{R_i}|V_t(\mathbf{e}_{d+1}\cdot\mathbf{Rx}) -V_t(\mathbf{e}_{d+1}\cdot\mathbf{Rx}_i) |d\omega_d(\mathbf{x}) \nonumber \\
   &&=  \int_{R_i}\left|\int_{\vartheta_i}^{\vartheta} \frac{d}{ds}V_t(\cos s)ds \right|d\omega_d(\mathbf{x})
  \leq  \omega_d(R'_i)\int_{\theta_i^{-}}^{\theta_i^{+}} \left|\frac{d}{ds}V_t(\cos s)\right|ds, \label{eq11}
\end{eqnarray}
where $\theta_i^{-},\theta_i^{+}$ are the low and high values for $\theta$ in the region $R'_i$.

The next step is to cover the hemisphere $\mathbb{S}^d_{\my_0}$ with overlapping ``bands".
Let $\gamma:=\bar{\theta}'-\underline{\theta}'$ and $N=\lfloor\frac{\gamma}{\|\mathcal{R}'\|}\rfloor$, where $\bar{\theta}', \underline{\theta}'$ are the high and low values for $\theta$ in $\mathbb{S}^d_{\my_0}$.
For  $k=1,\ldots,N$, let $J_k:=[(k-1)\gamma/N,k\gamma/N]$, and for $k=1,\ldots,N-1$, define $B_k$ to be all $\mathbf{x}\in\mathbb{S}^d_{\my_0}$ with its polar angle $\theta_{d-1}\in J_k\cup J_{k+1}$.
The common length for each interval is $\gamma/N\geq\|\mathcal{R'}\|\geq \diam(R'_i)\geq \theta_i^+-\theta_i^{-}$.
Thus, if $\theta^{-}_i\in J_k$, then $[\theta_i^{-},\theta_i^+]\subseteq J_k\cup J_{k+1}$ or, when $k=N-1$,
$[\theta_i^{-},\theta_i^+]\subseteq J_{N-1}\subseteq J_{N-2}\cup J_{N-1}$.
It follows that when $\theta_i^{-}\in J_k$, we have $R'_i\subseteq B_k$ and that the bound in (\ref{eq11}) may be replaced by
\begin{eqnarray*}
  \mathcal{E}_i
  &\leq & \omega_d(R'_i)\int_{(k-1)\gamma/N}^{(k+1)\gamma/N} \Big|\frac{d}{ds}V_t(\cos s)\Big|ds.
\end{eqnarray*}

Note that the $\{R'_i\}$ are nonoverlapping (no common interior point), so that $\sum_{R'_i\subset B_k}\omega_d(R'_i)\leq \omega_d(B_k)$.
Thus
$$\sum_{\{i:R'_i\subset B_k\}} \mathcal{E}_i
  \leq  \omega_d(B_k)\int_{(k-1)\gamma/N}^{(k+1)\gamma/N} \Big|\frac{d}{ds}V_t(\cos s)\Big|ds.$$

Since  $\omega_d(B_k)=|\mathbb{S}^{d-1}|\int_{(k-1)\gamma/N}^{(k+1)\gamma/N}|\sin s|^{d-1} ds$ and each $R'_i$ is contained in at least one band, there holds
\begin{eqnarray*}
  \sum_{i=1}^n \mathcal{E}_i &\leq& \sum_{k=1}^{N-1}\omega_d(B_k)\int_{(k-1)\gamma/N}^{(k+1)\gamma/N} \Big|\frac{d}{ds}V_t(\cos s)\Big|ds \\
   &\leq& |\mathbb{S}^{d-1}|\sum_{k=1}^{N-1}\int_{(k-1)\gamma/N}^{(k+1)\gamma/N}|\sin s|^{d-1} ds\int_{(k-1)\gamma/N}^{(k+1)\gamma/N} \Big|\frac{d}{ds}V_t(\cos s)\Big|ds=:\mathcal{E}_0.
\end{eqnarray*}

By  Lemma \ref{mhaskar2001} with $\delta=(2\|\mathcal{R}'\|)^{-1}<N/\gamma$ (because $N\leq \gamma/\|\mathcal{R}'\|\leq N+1$), $W(s)=|\sin s|^{d-1}$, $q=2$, we further have
\begin{eqnarray*}
  \sum_{i=1}^n \mathcal{E}_i\leq\mathcal{E}_0
  &\leq & \frac{32|\mathbb{S}^{d-1}|}{\delta}\sum_{k=1}^{N-1}\int_{(k-1)\gamma/N}^{(k+1)\gamma/N}  \Big|\frac{d}{ds}V_t(\cos s)\Big|W_\delta(s) ds\\
  &\leq & \frac{32|\mathbb{S}^{d-1}|}{\delta}\int_{\underline{\theta}'}^{\bar\theta'}  \Big|\frac{d}{ds}V_t(\cos s)\Big|W_\delta(s) ds=:\mathcal{E}_{00}.
\end{eqnarray*}

Both $W$ and $\Big|\frac{d}{ds}V_t(\cos s)\Big|$ are even in $s$, $\frac{d}{ds}V_t(\cos s)$ is a trigonometric polynomial.
Then, by (\ref{eq:W2}) and (\ref{eq:W3}), and our assumption $\delta=(2\|\mathcal{R}'\|)^{-1}\geq 2t-1$, we have
\begin{eqnarray*}
 && \int_{\underline{\theta}'}^{\bar\theta'}  \Big|\frac{d}{ds}V_t(\cos s)\Big|W_\delta(s) ds\\
  & \leq& \int_{0}^{\pi}  \Big|\frac{d}{ds}V_t(\cos s)\Big|W_\delta(s) ds
  =\frac{1}{2}\int_{-\pi}^{\pi}  \Big|\frac{d}{ds}V_t(\cos s)\Big|W_\delta(s) ds \\
  &\leq &    \frac{c_1}{2}\int_{-\pi}^{\pi}  \Big|\frac{d}{ds}V_t(\cos s)\Big|W(s) ds
  \leq  \frac{c_1c_2t}{2} \int_{-\pi}^{\pi} |V_t(\cos s)|W(s) ds\\
  &\leq&  c_1c_2t \int_{0}^{\pi} |V_t(\cos s)|(\sin s)^{d-1} ds.
\end{eqnarray*}
Since $\delta=(2\|\mathcal{R}'\|)^{-1}$, we further obtain
\begin{eqnarray*}
\sum_{i=1}^n \mathcal{E}_i\leq \mathcal{E}_{00}
&\leq& \frac{32|\mathbb{S}^{d-1}|}{\delta}c_1c_2(2t-1)\|V_t\|_{\mathbb{L}_1}
\leq c' (2t-1)\|\mathcal{R}\|C_v,
\end{eqnarray*}
where  $c':= 32|\mathbb{S}^{d-1}|c_1c_2$ and $C_v=\sup_{t\geq 1}\|V_t\|_{\mathbb{L}_1}$.  The proof is completed.
\end{proof}

Now, we prove the $\mathbb{L}_1$ Marcinkiewicz-Zygmund inequalities on the hemisphere for even and odd spherical polynomials.
\begin{theorem}\label{th:MZ}
Let $\mathcal{R}=\{R_1,\ldots,R_n\}$ be an equal area partition  of the hemisphere $\mathbb{S}^d_+$ and $\mathbf{x}_i\in R_i$, $i=1,\ldots,n$.
If $\|\mathcal{R}\|\leq \frac{\mu}{4t-2}\min\{\frac{1}{ c' C_v},\frac{1}{2}\}$ for some $t\in\mathbb{N}$ and $\mu\in(0,1)$, where $C_v=\sup_{t\geq 1}\|V_t\|_{\mathbb{L}_1}$, $c'$ is given in Lemma  \ref{pro:bound},
then
$$\bigg|\int_{\mathbb{S}^d_+}|P(\mathbf{x})|d\omega_d(\mathbf{x})-\frac{|\mathbb{S}_+^{d}|}{n}\sum_{i=1}^n |P(\mathbf{x}_i)|\bigg|\leq \mu\int_{\mathbb{S}_+^d}|P(\mathbf{x})|d\omega_d(\mathbf{x}), \ \forall P\in\mathbb{P}^o_{t}\cup\mathbb{P}^e_{t}.$$
\end{theorem}

\begin{proof}
We observe that $\min\{1,\frac{2}{ c' C_v}\}\leq 1,$ which implies $4t-2\leq \mu(2\|\mathcal R\|)^{-1}$.
Thus, for any $P\in\mathbb{P}^e_{t}$, we have
\begin{eqnarray*}
    &&\bigg|\int_{\mathbb{S}^d_+}|P(\mathbf{x})|d\omega_d(\mathbf{x})-\frac{|\mathbb{S}^d_+|}{n}\sum_{i=1}^n |P(\mathbf{x}_i)|\bigg|\\
    &\leq& \sum_{i=1}^{n}\int_{R_i} |P(\mathbf{x})-P(\mathbf{x}_i)|d\omega_d(\mathbf{x})\\
    &=& \sum_{i=1}^{n}\int_{R_i} \Big|\int_{\mathbb{S}^d}(V_t(\mz\cdot\mathbf{x}) -V_t(\mathbf{z}\cdot\mathbf{x}_i)) P(\mathbf{z})d\omega_d(\mathbf{z})\Big|d\omega_d(\mathbf{x})\\
  &\leq&  \int_{\mathbb{S}^d}|P(\mathbf{x})|d\omega_d(\mathbf{x})\sup_{\mathbf{z}\in\mathbb{S}^d}\sum_{i=1}^{n}\int_{R_i}|V_t(\mathbf{z}\cdot\mathbf{x}) -V_t(\mathbf{z}\cdot\mathbf{x}_i) |d\omega_d(\mathbf{x})\\
   &\leq& 2\int_{\mathbb{S}_+^d}|P(\mathbf{x})|d\omega_d(\mathbf{x})32|\mathbb{S}^{d-1}| c'_\theta (2t-1)\|\mathcal{R}\|C_v\\
   &\leq& \mu\int_{\mathbb{S}_+^d}|P(\mathbf{x})|d\omega_d(\mathbf{x}),
\end{eqnarray*}
where 
the first equality follows from (\ref{Vt}), the fourth inequality follows from Lemma \ref{pro:bound} and the last inequality follows from $\|\mathcal{R}\|\leq \frac{\mu}{4t-2}\min\{\frac{1}{c' C_v},\frac{1}{2}\}$.
The proof is completed.
\end{proof}

In the following, we apply Theorem \ref{th:MZ} to the Riemannian gradients of odd and even spherical polynomials.

\begin{corollary}\label{coro:MZP}
There is  $s_d>0$  such that
for each equal area partition $\mathcal{R}=\{R_1,\ldots,R_n\}$ of the hemisphere $\mathbb{S}^d_+$  with
  $\|\mathcal{R}\|<\frac{s_d}{t+1}$, each collection of point $\mathbf{x}_i\in R_i$, $i=1,\ldots,n$, and each polynomial
  $P\in\mathbb{P}^e_{t}\cup\mathbb{P}^o_{t}$, there holds
\begin{equation}\label{MZgrad}
\frac{1}{3\sqrt{d}}\int_{\mathbb{S}^d_+}\|\mathtt{grad} P(\mathbf{x})\|d\omega_d(\mathbf{x})\leq \frac{|\mathbb{S}^d_+|}{n} \sum_{i=1}^n \|\mathtt{grad} P(\mathbf{x}_i)\|
\leq  3\sqrt{d}\int_{\mathbb{S}^d_+}\|\mathtt{grad} P(\mathbf{x})\|d\omega_d(\mathbf{x}).
\end{equation}
\end{corollary}
\begin{proof}
For any $P\in\mathbb{P}_t^e$, we observe that  $\mathtt{grad} P\in\mathbb{P}_{t+1}^{o}$ and
$$\|\mathtt{grad} P\|=\sqrt{(P_1^o(\mx))^2+\cdots+(P_{d+1}^o(\mx))^2},$$
where $P_i^o$ is the $i$th element of $\mathtt{grad} P$.

Since
$$\frac{1}{\sqrt{d+1}}\sum_{i=1}^{d+1}|P^o_{i}(\mx_j)|\leq \sqrt{\sum_{i=1}^{d+1}(P_i^o(\mx_j))^2}\leq \sum_{i=1}^{d+1}|P^o_{i}(\mx_j)|,\quad j=1,\ldots,n,$$
by Theorem \ref{th:MZ}, we obtain
\begin{eqnarray*}
  \frac{|\mathbb{S}^d_+|}{n} \sum_{j=1}^{n}\|\mathtt{grad} P(\mx_j)\| &\leq&\frac{|\mathbb{S}^d_+|}{n} \sum_{j=1}^{n}\sum_{i=1}^{d+1}|P^o_{i}(\mx_j)|\leq (1+\mu)\sum_{i=1}^{d+1}\int_{\mathbb{S}^d_+}|P^o_i(\mx)|d\omega_d(\mx) \\
   &\leq&  (1+\mu)\sqrt{d+1}\int_{\mathbb{S}^d_+}\|\mathtt{grad} P(\mx)\|d\omega_d(\mx)\\
   &\leq&
    3\sqrt{d}\int_{\mathbb{S}^d_+}\|\mathtt{grad} P(\mx)\|d\omega_d(\mx),
\end{eqnarray*}
and
\begin{eqnarray*}
  \frac{|\mathbb{S}^d_+|}{n} \sum_{j=1}^{n}\|\mathtt{grad} P(\mx_j)\| &\geq&\frac{1}{\sqrt{d+1}}\frac{|\mathbb{S}^d_+|}{n} \sum_{j=1}^{n}\sum_{i=1}^{d+1}|P^o_{i}(\mx_j)|\\
  &\geq& \frac{1-\mu}{\sqrt{d+1}}\sum_{i=1}^{d+1}\int_{\mathbb{S}^d_+}|P^o_i(\mx)|d\omega_d(\mx)\\
   &\geq&  \frac{1-\mu}{\sqrt{d+1}}\int_{\mathbb{S}^d_+}\|\mathtt{grad} P(\mx)\|d\omega_d(\mx)\\
   &\geq&\frac{1}{3\sqrt{d}}\int_{\mathbb{S}^d_+}\|\mathtt{grad} P(\mx)\|d\omega_d(\mx).
\end{eqnarray*}
Thus, (\ref{MZgrad}) holds for any $P\in\mathbb{P}_t^e$.
Since for any $P\in\mathbb{P}_t^{*}$,   $\mathtt{grad} P\in\mathbb{P}_{t+1}^e$, by a similar proof, (\ref{MZgrad}) holds for any $P\in\mathbb{P}_t^{*}$.
The proof is completed.
\end{proof}

\section{Optimal asymptotic bounds for hemispherical $t$-designs}
In this section, we apply the Marcinkiewicz-Zygmund inequalities on $\mathbb{S}_+^d$ to establish
optimal asymptotic bounds for hemispherical $t$-designs.

We  define the set
\begin{equation}\label{def:omega}
\Omega:=\Bigg\{P \in\mathbb{P}_t^{e,*}:\frac{1}{|\mathbb{S}_+^d|}\int_{\mathbb{S}_+^d}\|\mathtt{grad} P(\mathbf{x})\|d\omega_d(\mathbf{x})<1 \Bigg\},
\end{equation}
where $\mathbb{P}_t^{e,*}:=\{P\in\mathbb{P}^e_t:\int_{\mathbb{S}^d_+}P(\mx)d\omega_d(\mx)=0\}$.

Following \cite{bondarenko2013},  let $U:\mathbb{P}_t^*\times\mathbb{S}^d\rightarrow \mathbb{S}^d$ be defined as follows
$$U(P,\mathbf{x})=
          \frac{\mathtt{grad} P(\mathbf{x})}{h_\epsilon(\|\mathtt{grad} P(\mathbf{x})\|)},
$$
where
 $ h_{\epsilon}(u)=\left\{\begin{array}{cl}
                           u & \mathrm{if}\, u>\epsilon \\
                           \epsilon & \mathrm{otherwise},
                         \end{array}\right.$
and $\epsilon=\frac{1}{6\sqrt{d}}$.

In the rest of this section,
let $n\geq C_{d} t^d$ be fixed for any $t\in\mathbb{N}$,  where $C_{d}>\big(\frac{54d c_{d}}{s_d}\big)^2$, $c_{d}$, $s_d$ are given in Theorem \ref{theo:eqb} and Corollary \ref{coro:MZP}, respectively.
Let $\mathcal{R}=\{R_1,\ldots,R_n\}$ be an equal area partition of  $\mathbb{S}^d_+$ with
$\|\mathcal{R}\|\leq c_{d} n^{-\frac{1}{d}}<\frac{s_d}{54dt}$.
Let $\mathbf{y}^0_i\in R_i$, $i=1,\ldots,n$ be arbitrarily.

For each $i=1,\ldots,n$, let $y_i:\mathbb{P}_t^*\times[0,\infty)\rightarrow\mathbb{S}^d$ be a mapping satisfying the following differential equation
\begin{equation}\label{eq:diff}
 \frac{d}{ds}y_i(P,s) = U(P,y_i(P,s)),
\end{equation}
with the initial condition
\begin{equation}\label{x0}
y_i(P,0)=\mathbf{y}^0_i,\quad \forall P\in\mathbb{P}_t^*.
\end{equation}

 Note that  $U(P,y)$ is Lipschitz continuous in both $P$ and $y$.
For each $i$, the ODE system (\ref{eq:diff})-(\ref{x0}) has a continuous solution, which is continuous in both $P$ and $s$.

Adopting the construction of a continuous mapping $F$ in \cite{bondarenko2013}, we prove the following lemma.
Firstly, let $T:\mathbb{S}^d\rightarrow\mathbb{S}_+^d$ be defined as
\begin{equation}\label{defT}
 T(\mx)=\left\{\begin{array}{cl}
             \mx & \mathrm{if}\ \mx\cdot\me_{d+1}\geq0 \\
             -\mx & \mathrm{if}\ \mx\cdot\me_{d+1}<0.
           \end{array}\right.
\end{equation}
Obviously, $P(T(\mx))=P(\mx)$, $\forall P\in\mathbb{P}_t^e$ and $\forall \mx\in\mathbb{S}^d$.

\begin{lemma}\label{lem:ineq}
Let $F:\mathbb{P}_t^*\rightarrow (\mathbb{S}^d_+)^n$ be the mapping given by
\begin{equation}\label{map:F}
 F(P):=\big(T(y_1(P,\tfrac{s_d}{3t})),\ldots,T(y_n(P,\tfrac{s_d}{3t}))\big).
\end{equation}
Then,
$\frac{1}{n}\sum_{i=1}^{n}P(T(y_i(P,\frac{s_d}{3t})))>0$ for any  $P\in \partial\Omega$.
\end{lemma}

\begin{proof}
We  fix $P\in \partial\Omega=\{P \in\mathbb{P}_t^{e,*}:\frac{1}{|\mathbb{S}_+^d|}\int_{\mathbb{S}^d_+}\|\mathtt{grad} P(\mathbf{x})\|d\omega_d(\mathbf{x})=1 \}$.
For notational simplicity, we write $y_i(s)$ in place of $y_{i}(P,s)$, $i=1,\ldots,n.$

Since $P\in\mathbb{P}_t^e$, by (\ref{defT}) we have
$$P(T(y_i(\tfrac{s_d}{3t})))=P(y_i(\tfrac{s_d}{3t})),\quad i=1,\ldots,n.$$
By the Newton-Leibniz formula, we obtain
\begin{equation}\label{NL}
\frac{1}{n}\sum_{i=1}^{n}P(T(y_i(\tfrac{s_d}{3t})))=\frac{1}{n}\sum_{i=1}^{n}P(y_i(\tfrac{s_d}{3t}))
 = \frac{1}{n}\sum_{i=1}^{n}P(\mathbf{y}^0_i)+\int_{0}^{\frac{s_d}{3t}} \frac{d}{ds}\Big(\frac{1}{n}\sum_{i=1}^{n}P(y_i(s))\Big)ds.
\end{equation}
In the following, we estimate the first and second terms on the right-hand side of equation (\ref{NL}), respectively.

(i) We first observe that
\begin{eqnarray*}
      \bigg|\frac{1}{n}\sum_{i=1}^{n}P(\mathbf{y}^0_i)\bigg|
      & =&\bigg|\frac{1}{|\mathbb{S}^d_+|}\sum_{i=1}^{n}\int_{R_i}P(\mathbf{y}^0_i)-P(\mathbf{y})d\omega_d(\mathbf{y})\bigg|\\
      &\leq& \frac{1}{|\mathbb{S}^d_+|}\sum_{i=1}^{n}\int_{R_i}|P(\mathbf{y}^0_i)-P(\mathbf{y})|d\omega_d(\mathbf{y})\\
      &\leq& \frac{\|\mathcal{R}\|}{n}\sum_{i=1}^{n}\max_{\{\mathbf{z}_i\in \mathbb{S}^d_+:\dist(\mathbf{z}_i,\mathbf{y}^0_i)\leq \|\mathcal{R}\|\}}\|\mathtt{grad} P(\mathbf{z}_i)\|,
\end{eqnarray*}
where the first equality follows from $\int_{\mathbb{S}^d_+}P(\mathbf{y})d\omega_d(\mathbf{y})=0$.
Hence, for $\bar{\mathbf{z}}_i\in \mathbb{S}^d_+$ such that $\dist(\bar{\mathbf{z}}_i,\mathbf{y}^0_i)\leq \|\mathcal{R}\|$ and
$$\|\mathtt{grad} P(\bar{\mathbf{z}}_i)\|=\max_{\{\mathbf{z}_i\in\mathbb{S}^d_+:\dist(\mathbf{z}_i,\mathbf{y}^0_i)\leq \|\mathcal{R}\|\}}\|\mathtt{grad} P(\mathbf{z}_i)\|,$$
we obtain
$$\bigg|\frac{1}{n}\sum_{i=1}^{n}P(\mathbf{y}^0_i)\bigg|\leq \frac{\|\mathcal{R}\|}{n}\sum_{i=1}^{n}\|\mathtt{grad} P(\bar{\mathbf{z}}_i)\|.$$

Let $\mathcal{R}'=\{R_1',\ldots,R'_n\}$ with $R_i'=R_i\cup\{\bar{\mathbf{z}}_i\}$ be another equal area partition.
Clearly, $\|\mathcal{R}'\|\leq2\|\mathcal{R}\|$ and we obtain $\|\mathcal{R}'\|\leq \frac{s_d}{27dt}$.
By Corollary \ref{coro:MZP}, we obtain
\begin{equation}\label{ineq1}
  \bigg|\frac{1}{n}\sum_{i=1}^{n}P(\mathbf{y}^0_i)\bigg|<  \|\mathcal{R}\|\frac{3\sqrt{d}}{|\mathbb{S}^d_+|}\int_{\mathbb{S}_+^d}\|\mathtt{grad} P(\mathbf{y})\|d\omega_d(\mathbf{y})<\frac{ s_d}{18\sqrt{d}t},\quad \forall P\in\partial\Omega.
\end{equation}

(ii)
By (\ref{eq:diff}), we have
\begin{equation}\label{indf3}
\frac{d}{ds}\Big(\frac{1}{n}\sum_{i=1}^{n}P(y_i(s))\Big)
\geq
 \frac{1}{n}\sum_{i=1}^{n}\frac{\|\mathtt{grad} P(y_i(s))\|^2}{h_\epsilon(\|\mathtt{grad} P(y_i(s))\|)}\nonumber\\
 \geq \frac{1}{n}\sum_{i=1}^{n}\|\mathtt{grad} P(y_i(s))\|-\epsilon.
\end{equation}

Since $\|U(P,\mathbf{y})\|\leq 1$ for each $\mathbf{y}\in\mathbb{S}^d$,  by (\ref{eq:diff}) we have
$\|\frac{d}{ds}y_i(s)\|\leq1$ for every $s$.
Hence, we obtain $\dist(\mathbf{y}^0_i,y_i(s))\leq s$, $i=1,\ldots,n.$

Now for $s\in[0,\frac{s_d}{3t}]$, let $\mathcal{R}''=\{R''_1,\ldots,R''_n\}$ with $R''_i=R_i\cup\{y_i(s)\}$.
Then $$\|\mathcal{R}''\|\leq \frac{s_d}{54dt}+\frac{s_d}{3t}.$$
By Corollary \ref{coro:MZP} and (\ref{indf3}), for any $P\in\partial\Omega$ and  every $s\in[0,\frac{s_d}{3t}],$ we have
\begin{eqnarray}
  \frac{d}{ds}\Big(\frac{1}{n}\sum_{i=1}^{n}P(y_i(s))\Big)
  & \geq&\frac{1}{n}\sum_{i=1}^n\|\mathtt{grad} P(y_i(s))\|-\epsilon  \nonumber\\
   &>&  \frac{1}{3\sqrt{d}}\frac{1}{|\mathbb{S}^d_+|}\int_{\mathbb{S}_+^2}\| \mathtt{grad} P(y(s))\|d\omega_d(\mathbf{x})- \frac{1}{6\sqrt{d}} \nonumber\\
    &>& \frac{1}{6\sqrt{d}}.\label{indf2}
\end{eqnarray}

Finally, (\ref{NL}), (\ref{ineq1}) and (\ref{indf2}) imply
\begin{equation*}
\frac{1}{n}\sum_{i=1}^{n}P(T(y_i(\tfrac{s_\theta}{3t})))=\frac{1}{n}\sum_{i=1}^{n}P(y_i(\tfrac{s_\theta}{3t}))
  >\frac{1}{6\sqrt d}\frac{s_d}{3t}-\frac{ s_d}{18\sqrt{d}t}
  =0.
\end{equation*}
The proof is completed.
\end{proof}

Now, we prove the asymptotic bounds for hemispherical $t$-designs.
\begin{theorem}
 Let $F$ and $\Omega$ be defined as  (\ref{map:F}) and (\ref{def:omega}), respectively.
  Then, there is $\bar{P}\in\mathbb{P}_t^*$ such that $(T(y_1(\bar{P},\frac{s_d}{3t})),\ldots,T(y_n(\bar{P},\frac{s_d}{3t})))$ forms a hemispherical $t$-design on $\mathbb{S}^d_+$.
\end{theorem}
\begin{proof}
Let $\bar{U}:(\mathbb{S}^d_+)^n\rightarrow \mathbb{P}^*_t$ be a mapping defined by
$$\bar{U}(\mathbf{y}_1,\ldots,\mathbf{y}_n)= Q_{\mathbf{y}_1}+\ldots+Q_{\mathbf{y}_n},$$
where $Q_{\mathbf{y}_i}$ are polynomials of the form  (\ref{poly:Q}),
and let $f=\bar{U}\circ F:\mathbb{P}_t^*\rightarrow\mathbb{P}_t^*$.
By (\ref{poly:Q1}) and Lemma \ref{lem:ineq},
$$\langle P,f(P) \rangle=\sum_{i=1}^{n}\langle P,Q_{T(y_i(P,\frac{s_d}{3t}))} \rangle=\sum_{i=1}^{n}P(T(y_i(P,\tfrac{s_d}{3t})))>0,\quad \forall P\in \partial\Omega.$$
Thus by Theorem  \ref{lemm:mapp}, there is $\bar{P}\in\Omega\subseteq\mathbb{P}_t^*$ such that $f(\bar{P})=\sum_{i=1}^{n}Q_{T(y_i(\bar{P},\frac{s_d}{3t}))}=0$.
Furthermore, by (\ref{poly:Q1}) we have
$$\langle P, f(\bar{P})\rangle
=\sum_{i=1}^{n}\langle P,Q_{T(y_i(\bar{P},\frac{s_d}{3t}))} \rangle=\sum_{i=1}^{n}P(T(y_i(\bar{P},\tfrac{s_d}{3t})))=0.
\quad \forall P\in\mathbb{P}_t^*.$$
By Corollary \ref{coro:con},  $(T(y_1(\bar{P},\frac{s_d}{3t})),\ldots,T(y_n(\bar{P},\tfrac{s_d}{3t})))$ forms a hemispherical $t$-design on $\mathbb{S}^d_+$.
The proof is completed.
\end{proof}

\section{Examples}
In this example, we first give  a hemispherical 3-design $\mathcal{X}_8$ with 8 points.
Specifically,
\begin{gather*}
\mathcal{X}_8=\left\{\left[
   \begin{array}{c}
    \sqrt{\frac{4-\sqrt{3}}{6}} \\
    0  \\
     \frac{\sqrt{3}+3}{6}\\
   \end{array}
 \right],
\left[
   \begin{array}{c}
    0 \\
    \sqrt{\frac{4-\sqrt{3}}{6}}  \\
     \frac{\sqrt{3}+3}{6} \\
   \end{array}
 \right]
,
\left[
   \begin{array}{c}
    -\sqrt{\frac{4-\sqrt{3}}{6}} \\
    0 \\
     \frac{\sqrt{3}+3}{6} \\
   \end{array}
 \right]
,
\left[
   \begin{array}{c}
    0 \\
    -\sqrt{\frac{4-\sqrt{3}}{6}}  \\
     \frac{\sqrt{3}+3}{6} \\
   \end{array}
 \right]
 ,\right.
 \\ \qquad \qquad \qquad
 \left.
 \left[
   \begin{array}{c}
    \sqrt{\frac{4+\sqrt{3}}{12}} \\
    \sqrt{\frac{4+\sqrt{3}}{12}} \\
     \frac{3-\sqrt{3}}{6}\\
   \end{array}
 \right],
\left[
   \begin{array}{c}
    -\sqrt{\frac{4+\sqrt{3}}{12}} \\
    \sqrt{\frac{4+\sqrt{3}}{12}}  \\
     \frac{3-\sqrt{3}}{6} \\
   \end{array}
 \right]
,
\left[
   \begin{array}{c}
    -\sqrt{\frac{4+\sqrt{3}}{12}} \\
    -\sqrt{\frac{4+\sqrt{3}}{12}}  \\
     \frac{3-\sqrt{3}}{6} \\
   \end{array}
 \right],
 \left[
   \begin{array}{c}
    \sqrt{\frac{4+\sqrt{3}}{12}} \\
    -\sqrt{\frac{4+\sqrt{3}}{12}}  \\
     \frac{3-\sqrt{3}}{6} \\
   \end{array}
 \right]
\right\}.
\end{gather*}
We show $\mathcal{X}_8$ in Fig. \ref{fii}.
Let $R\mathcal{X}_8$ be the hemispherical 3-design obtained by the rotation
$\mathbf{R}=\left[
   \begin{array}{ccc}
    \frac{\sqrt{2}}{2} & -\frac{\sqrt{2}}{2} & 0 \\
    \frac{\sqrt{2}}{2} & \frac{\sqrt{2}}{2} & 0  \\
     0& 0& 1 \\
   \end{array}
 \right]$ of $\mathcal{X}_8$.
We observe that the set $\mathcal{X}_{16}=R\mathcal{X}_8 \cup \mathcal{X}_8$ is also a hemispherical 3-design.
The number of points in $\mathcal{X}_{16}$ is consistent with $(t+1)^2$ in \cite{chen2011} with $t=3$.
We show $\mathcal{X}_{16}$ in Fig. \ref{fii}.

Since the hemisphere is rotationally invariant about the $z$-axis.
Thus, the set obtained by  any rotation of $\mathcal{X}_8$ or $\mathcal{X}_{16}$ about the $z$-axis is  a hemispherical 3-design.

\begin{figure}[h]
  \centering
  \includegraphics[width=5.cm]{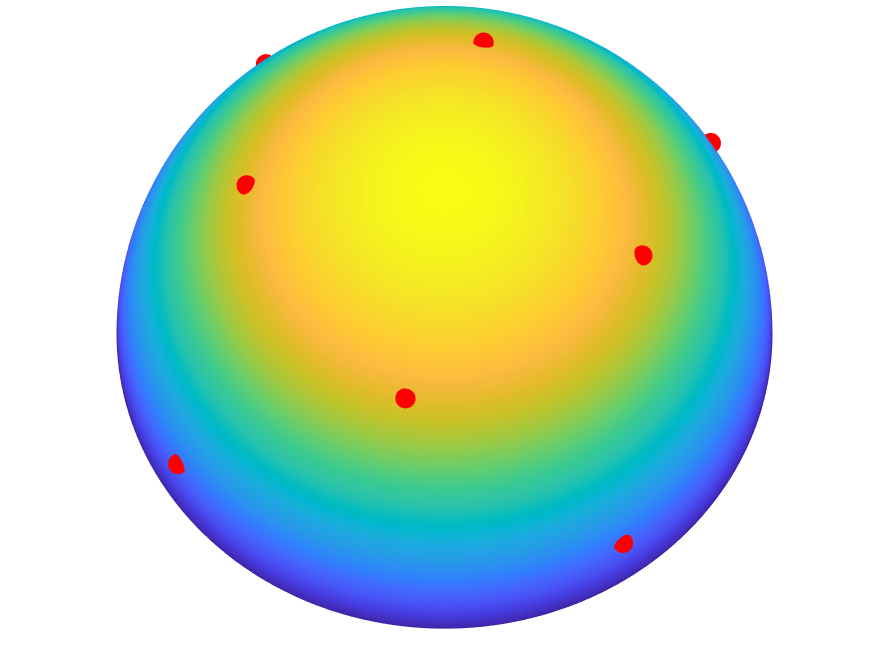}\qquad
  \includegraphics[width=5.cm]{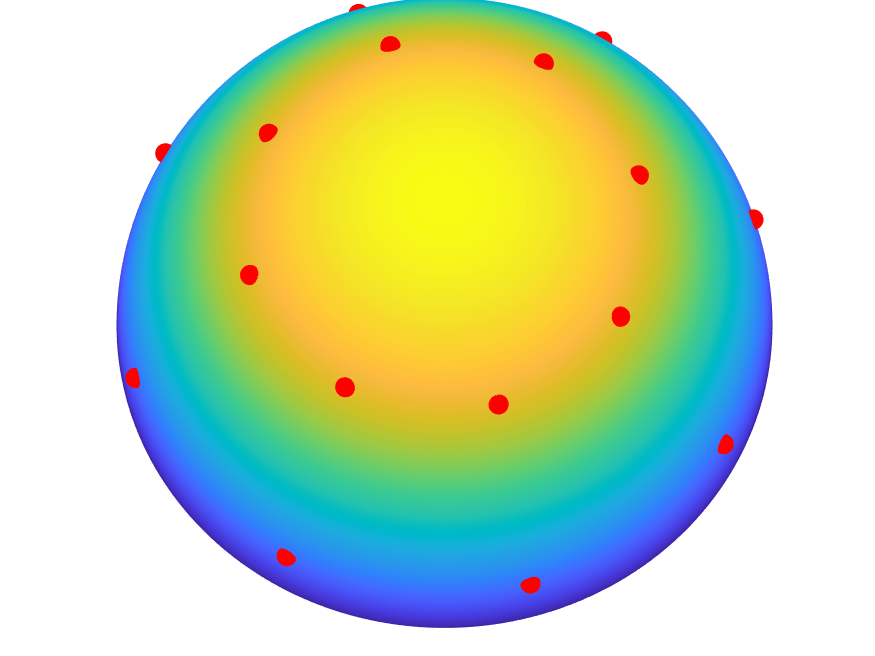}
  \caption{Hemispherical 3-design with 8 points (left) and 16 points (right).}\label{fii}
\end{figure}

\section{Conclusion}
In this work, we study hemispherical $t$-designs and $\mathbb{L}_1$ Marcinkiewicz-Zygmund inequalities on the hemisphere.
Firstly, we  define a set of new spherical polynomials $\{B^d_{\ell,k}:k=1,\ldots,2\ell+1,\ell=0,1,\ldots,t\}$ on $\mathbb{S}^d_+$ and  present an equivalent condition (\ref{con}) for   hemispherical $t$-designs.
Next, we apply the recursive zonal equal area partition method  \cite{leopardi2006} to show that there is a recursive zonal equal area partition of the hemisphere with bounded partition norm.
Then, we prove the  $\mathbb{L}_1$ Marcinkiewicz-Zygmund inequalities on the hemisphere for even and odd spherical polynomials.
Finally, we apply the $\mathbb{L}_1$ Marcinkiewicz-Zygmund inequalities and the new set of polynomials on the hemisphere to  establish asymptotic bounds for hemispherical $t$-designs.

\bibliographystyle{spmpsci}      
\bibliography{sn-bibliography}

\end{document}